\documentclass[runningheads,envcountsame,numbook]{svjour2-}
\smartqed  % flush right qed marks, e.g. at end of proof
\usepackage{graphicx}
\usepackage{mathptmx}      % use Times fonts if available on your TeX system
\usepackage{amsmath,amssymb,latexsym,xcolor}
\newcommand{\Aut}[1]{\mathsf{Aut}(#1)}%
\journalname{%
{\color{blue} Submitted version} 
compiled~{\color{blue}June 18, 2007}
in {\color{blue} Frankfurt (Main), Germany}
}
\renewcommand{\emph}{\textbf}

\newcommand{\CAT}[1]{\mathsf{CAT}(#1)}%
\def\Trans#1#2_#3.{\mathsf{Trans}_{#3}(#1,\,#2)}%
\def\Aut#1_#2.{\mathsf{Aut}_{#2}(#1)}%
\def\Stab#1_#2.{#2_{\{#1\}}}%
\def\Fix#1_#2.{{#2_{#1}}}%%% \mathsf{Fix}_{#2}(#1)}%

\newcommand{\diag}{\operatorname{diag}}%

\begin{document}

\title{Contraction groups in complete Kac-Moody groups
\thanks{This work was supported by Australian Research Council grant DP0556017. %
The second author thanks
the Centre de Recerca Matem\'atica for its %
hospitality and support during June 2007.       
The second and third authors thank
The University of Newcastle %(Australia)
for its
hospitality and support.%
}
}
%Grants or other notes

%\subtitle{}

%\titlerunning{Commensurators of neighborhoods}        % if too long for running head

\author{U. Baumgartner   \and
              J. Ramagge \and
              B. R\'emy    %etc.
}


\institute{Udo Baumgartner \at
              School of Mathematical and Physical Sciences,
              The University of Newcastle, University Drive, Building V,
              Callaghan, NSW 2308, Australia, %\\
              Tel.: +61-2 4921 5546, %\\
              Fax: +61-2 4921 6898, %\\
              \email{Udo.Baumgartner@newcastle.edu.au}           %  \\
           \and
           Jacqui Ramagge           \at
           School of Mathematics and Applied Statistics
          University of Wollongong,
      Wollongong NSW 2522,
           Australia,
    Tel.: +61-2 4221 3845 (school office)
    Fax: +61-2 4221 4845
    %M: 0407 065 911
            \email{ramagge@uow.edu.au}
           \and
           Bertrand R\'emy \at
           Universit\'e de Lyon, Lyon, F-69003, France;
           Universit\'e de Lyon 1, Institut Camille Jordan, F-69622, France;
           CNRS, UMR 5208, Villeurbanne, F-69622, France, %\\
           Tel.: +33-4 72 44 82 07, %\\
           Fax: +33-4 72 43 16 87, %\\
           \email{remy@math.univ-lyon1.fr}
}

\date{Received: date / Revised: date}
\maketitle

\begin{abstract}
Let $G$ be
an abstract Kac-Moody group
over a finite field
and $\overline{G}$
the closure of
the image
of $G$
in the automorphism group
of its positive building.
We show that
if
the
Dynkin diagram
associated to $G$
is irreducible
and neither
of spherical nor of affine type,
then
the contraction groups
of elements
in $\overline{G}$
which are
not topologically periodic
are not closed.
(In those groups
there always exist elements
which are not topologically periodic.)
\subclass{
%                  22D05, % (General properties and structure of locally compact groups)
%                  22D45  % (Automorphism groups of locally compact groups)
%                  (primary)
%                  %
%                  20E36, % (General theorems concerning automorphisms of groups),
%                  (secondary)
                  contraction group
         \and topological Kac-Moody group
         \and totally disconnected, locally compact group}
\end{abstract}

%%%%%%%%%%%%%%% BODY OF TEXT %%%%%%%%%%%%%%%%%%%
%\pagebreak

\section{Introduction}\label{sec:intro}

%\paragraph{Intention}
%what are contraction groups for
Let $\mathfrak{g}$ be
a continuous automorphism
of a topological group $G$
with continuous inverse.
Its \emph{contraction group}
is the subgroup
of $G$
defined by
\[
U_{\!\mathfrak{g}}:=\bigl\{x\in G\colon \mathfrak{g}^n(x)\longrightarrow e\quad
\text{as $n$ goes to infinity}\bigr\}\,.
\]
Interest in contraction groups
%and related concepts
has been stimulated
by applications
in the theory of
probability measures
and
random walks on,
and
the representation theory of,
locally compact groups.
For these applications
it is important
to know
whether
a contraction group
is closed.
%a question
%which
%we also consider %changed
%here.
%
We refer
the reader
to the introduction in~\cite{contrG+scales(AUT(tdlcG))} %changed
and the references
cited there
for information
about
the applications of contraction groups
and known results.
Recent articles
which treat contraction groups
are
\cite{class(simple-factors)<comp-ser(tdcontrGs)}
and
\cite{contrLieGs(localFs)}.

The article~\cite{contrG+scales(AUT(tdlcG))}
studied
the contraction group $U_{\!\mathfrak{g}}$
%of a continuous automorphism
and its supergroup
\[
P_{\!\mathfrak{g}}:=
\bigl\{x\in G\colon \{\mathfrak{g}^n(x)\colon n\in\mathbb{N}\}\ \text{is relatively compact}\bigr\} %\,,
\]
in the case where
the ambient group
is locally compact
and
totally disconnected,
a case in which previously little was known.
In contrast to $U_{\!\mathfrak{g}}$,
the group
$P_{\!\mathfrak{g}}$
is always closed
if
the ambient group $G$
is totally disconnected
\cite[Proposition~3, parts~($\romannumeral3$) and~($\romannumeral2$)]{tdlcG.structure}.
The group $P_{\!\mathfrak{g}}$
was named
the \emph{parabolic group}
of the automorphism $\mathfrak{g}$
in~\cite{contrG+scales(AUT(tdlcG))}
because
for any
inner automorphism
of a semisimple algebraic group
over a local field
its parabolic group
is
the group
of rational points
of a rational parabolic subgroup
(and every such group is of that form);
the corresponding contraction group
in that case
is
the group of rational points
of the unipotent radical  
of the parabolic subgroup.
In this algebraic group context,
identifying parabolic subgroups
(in the dynamical sense,
introduced above) 
and their unipotent radicals 
with parabolic subgroups
(in the algebraic group sense) 
and the corresponding contraction groups 
is a crucial technique                     
used by G. Prasad
to prove strong approximation
for semisimple groups
in positive characteristic
\cite{strong-approx(ssimpleGs(funcFs))}. 
This technique 
was later used again 
by G. Prasad 
to give a simple proof 
of Tits's theorem 
on cocompactness 
of open non-compact subgroups 
in simple algebraic groups 
over local fields \cite{elem:BTR+T}, 
which can be proved also 
by appealing to 
Howe-Moore's property. 
%nowhere 

In this article
we
investigate which  
contraction groups
of inner automorphisms
in complete Kac-Moody groups
are closed.
Complete Kac-Moody groups
(which we 
 introduce in Section~\ref{sec:frame})
are combinatorial generalizations
of semisimple algebraic groups
over local fields.
In contrast to                                  
members of
the latter class of groups,
complete Kac-Moody groups
are
generically
\underline{non-linear},
totally disconnected, locally compact groups.
These properties
make them
perfect test cases
for the developing structure theory
of totally disconnected, locally compact groups
which was
established in~\cite{tdlcG.structure},
and further advanced
in~\cite{furtherP(s(tdG))} and~\cite{tidy<:commAut(tdlcG)}.

Our main result is
the following theorem,
in whose statement the contraction group
of a group element
$g$
is understood
to be the contraction group
of the inner automorphism $\mathfrak{g}\colon x\mapsto gxg^{-1}$.

\begin{theorem}[Main Theorem]
\label{thm:MainThm}
Let $G$ be
an abstract Kac-Moody group
over a finite field
and $\overline{G}$ be
the closure of
the image
of $G$
in the automorphism group
of its positive building.
Then the following are true:
\begin{enumerate}
\item
The contraction group
of any topologically periodic element
in $\overline{G}$
is trivial.
\item
If
the type of $G$
is irreducible
and neither
spherical nor affine,
then
the contraction group
of any element
that is
not topologically periodic
in $\overline{G}$
is not closed.
\end{enumerate}
Furthermore,
the group $\overline{G}$
contains
non-topologically periodic elements
whenever
$G$
is not of spherical type.
\end{theorem}

The second assertion of                         
Theorem~\ref{thm:MainThm} is
in sharp contrast
with the known results
about contraction groups of
elements in spherical and affine Kac-Moody groups.
In particular,
all contraction groups of inner automorphisms
are closed for semisimple algebraic
groups over local fields;
this follows from
the representation
of contraction groups
as rational points
of unipotent radicals
and we direct the reader to
part~2 of Proposition~\ref{prop:contrGs(spherical,known-affine)}
for a slightly more general statement.

Consequently, all contraction groups
of inner automorphisms are closed for certain affine
Kac-Moody groups, namely those that are
geometric completions of Chevalley group schemes over the
rings of Laurent polynomials over finite fields.
For completions
of Kac-Moody groups
of any spherical type
the same
is seen
to be true;
see                                     
part~1 of Proposition~\ref{prop:contrGs(spherical,known-affine)}.

Thus Theorem~\ref{thm:MainThm} 
and 
Proposition~\ref{prop:contrGs(spherical,known-affine)}
provide another instance
of the strong dichotomy
between Euclidean
and
arbitrary non-Euclidean buildings
with large automorphism groups
which is already evident in        
results
such as
the Simplicity Theorem
in~\cite{simpl+superrig(twin-building-lattices)}
and
the strong Tits alternative
for infinite irreducible Coxeter groups
by Margulis-Noskov-Vinberg%
~\cite{linG-virt-free-quot,strongTitsAlt(<CoxeterG)}.

%UDO: reordered and reformulated the remaining paragraphs, 
The groups covered by
the second part of our Main Theorem
are 
topologically simple%
~\cite{ts(Kac-Moody)+commensurator}, %
%UDO: for the following we would need to refer to Carbone, Ershov and Ritter 
indeed in many cases
 algebraically simple%
~\cite{CarErsRit} groups,  
whose 
flat rank
assumes 
all positive integral values~\cite{flat-rk(autGs(buildings))}, 
and indeed are 
the first known groups 
who have 
non-closed contraction groups 
and 
whose flat rank 
can be larger than~$2$; 
we refer the reader to%
~\cite{tidy<:commAut(tdlcG),flat-rk(autGs(buildings))}
for the definition of flat rank. 
They are thus 
`larger' 
but similar to 
the group 
of type-preserving isometries 
of a regular, locally finite tree,  
which is 
a 
simple, 
totally disconnected, 
locally compact 
group  
of flat rank~1, 
whose 
non-trivial contraction groups 
are non-closed. 
This follows from 
Example~3.13(2) in~\cite{contrG+scales(AUT(tdlcG))} 
and 
Remark~\ref{rem:subgroup-contractionGs+parabolics_geo}.

%UDO: I believe that this is now redundant 
%Geometrically, our examples
%differ from earlier ones in that
%they apply to groups acting on buildings
%which are neither spherical nor affine.

%	%UDO: suggest to omit this example all-together 
%	Examples of contraction groups
%	which are not closed are already known.
%	%
%	A special case of example~3.13(3) in%
%	~\cite{contrG+scales(AUT(tdlcG))}
%	is as follows;                                                %changed
%	%The following example already appeared
%	%in~\cite{tdlcG.structure}.
%	suppose $F$ is a finite group
%	and $G=F^\mathbb{Z}$.
%	Define a shift automorphism $\alpha$ on $G$ %UDO: omit `shift' 
%	via $\alpha(f)_n=f_{n-1}$.
%	%
%	Denoting by $e$ the identify of $F$,
%	the sets
%	\[
%	V_N=\{f\in F^\mathbb{Z}\colon f_{-N}=f_{-N+1}=\cdots=f_{N+1}=f_{N}=e\}
%	\]
%	for $N\in\mathbb{N}$
%	are compact open subgroups
%	of $G$ and
%	form a base of neighbourhoods of the identity.
%	%
%	The contraction group of $\alpha$ is
%	\[
%	U_\alpha=\{ f\in F^\mathbb{Z} \colon
%	\text{ there exists } N\in\mathbb{N}
%	\text{ such that } f_{-n}=e
%	\text{ for all } n>N\}
%	\]
%	which we can identify with $(\bigoplus_{n<0}F)\oplus(\prod_{n\geq 0}F)$.
%	%
%	The group $U_{\alpha}$ is not closed %UDO: omit rest of this sentence 
%	and the only subgroup
%	tidy for $\alpha$
%	is $G$ itself.

%
The Main Theorem
will be proved
within the wider framework
of groups 
with
a locally finite twin root datum.                       
Within
this wider framework
we need to impose
the additional assumption
that
the root groups
of the given root datum
are contractive
(a condition
introduced in Subsection~\ref{subsec:contr-rootGs})
in order
to be able
to prove
the analogue of
the second statement above.
In the Kac-Moody case
this condition
is automatically fulfilled
by a theorem
of Caprace and R\'emy.
In all cases,                        %changed
the geometry
of the underlying Coxeter complex
will
%enter
%decisively                         %changed
play a crucial role
in the proof
%of our Main Theorem,         %changed
via the existence
of `a fundamental hyperbolic configuration',
see Theorem~\ref{thm:fundamental-hyperbolic-configuration_line}.

\section{Framework}
\label{sec:frame}

% introduce the objects of our study
We study complete Kac-Moody groups; these were         %changed
introduced in~\cite{tG(Kac-Moody-type)+rangle-twinnings+:lattices}
under the name `topological Kac-Moody groups'.
A complete Kac-Moody group
is a geometrically defined                                   %changed
completion
of an abstract Kac-Moody group
over a finite field.
Every Kac-Moody group
is a group-valued functor,
$\mathbf{G}$ say,
on rings,
which is defined by
a Chevalley-Steinberg type presentation,
whose main parameter
is an integral matrix,
a `generalized Cartan matrix',
which
also defines
a Coxeter system
of finite rank;
see
\cite[Subsection~3.6]{unique+pres(Kac-MoodyG(F))}
and
\cite[Section~9]{GKac-Moody-depl+presque-depl}
for details.
For each ring $R$,
the value
$G:=\mathbf{G}(R)$
of the functor $\mathbf{G}$
on $R$ is an
\emph{abstract Kac-Moody group over $R$}.

For each field %MAYBE EVEN FOR EVERY ring, %changed
$R$ the
Chevalley-Steinberg presentation
endows                        %
the abstract Kac-Moody group $\mathbf{G}(R)$
with the structure
of a \emph{group with a twin root datum},
which is
%the way
%we will view
%the abstract Kac-Moody group $\mathbf{G}(R)$  %changed
%in this paper.
the context in which our results are stated.           %changed
A twin root datum
is
a collection $\bigl((U_\alpha)_{\alpha\in \Phi},H\bigr)$
of subgroups
of $G$
indexed by the set
$\Phi$ of roots of                                         %changed
the associated Coxeter system $(W,S)$
and satisfying certain axioms
which ensure that
the group $G$
acts on
a `twinned' pair of buildings
of type $(W,S)$;
see~\cite[1.5.1]{GKac-Moody-depl+presque-depl}.
See Subsection~0.3,
respectively~0.4,
in~\cite{simpl+superrig(twin-building-lattices)}
for the list of axioms
of a twin root datum
and references
to further literature
on twin root data
and twin buildings.

In order to define
the \emph{geometric completion}
of $\mathbf{G}(R)$,
assume that
%the ring
$R$
is a finite field.
Under this assumption
all the groups
which constitute
the natural root datum
of $\mathbf{G}(R)$
are finite;
groups
with a twin root datum
having this property
will be called
\emph{groups with a locally finite twin root datum}.
The \emph{Davis-realization}
of the buildings
defined by
a locally finite twin root datum
are locally finite,
metric,
$\CAT0$-complexes
in the sense of \cite{GMW319}
all of whose cells
have finite diameter;                                                 %changed
see~\cite[Section~1.1]{flat-rk(autGs(buildings))}
for a short explanation
following M. Davis' exposition in~\cite{buildings=CAT0}.
The \emph{geometric completion}
of a group
$G$
with locally finite twin root datum
is the closure
of the image
of $G$
in the
automorphism group
of the Davis-realization
of the positive building
defined by
the given root datum;
if $G$ is
an abstract Kac-Moody group
over a finite field
that completion
will be called
the corresponding  \emph{complete Kac-Moody group}
and denoted by~${\overline{G}}$.

The completion
of an abstract Kac-Moody group
is defined by %UDO: changed `was' to `is' 
its action on its building
and
our techniques 
rely on           
the $\CAT0$-geometry
of the building,
in particular
the action of the group
`at infinity'.
However,
note that
the topology
and the completion
of a group with locally finite twin root datum
do \underline{not}
depend on the $\CAT0$-structure,
only on
the combinatorics
of the action
on the building;
see Lemma~2 in~\cite{flat-rk(autGs(buildings))}.
Therefore
one should be able
to dispense with
the use of
the Davis-realization below.

We summarize
the basic topological properties of
automorphism groups
of locally finite complexes
in the following proposition.

\begin{proposition}\label{prop:topProp(Aut(lf-complex)}
Let $X$ be
a connected, locally finite cell complex.
Then
the com\-pact-open topology on
%the automorphism group of $X$,
$\Aut X_.$ %,
is a
locally compact,
totally disconnected (hence Hausdorff)
group topology.
This topology
has a countable basis,
hence is $\sigma$-compact and metrizable.
Stabilizers and fixators
of finite subcomplexes
of $X$
in $\Aut X_.$
are
compact, open subgroups
of $\Aut X_.$
and
the collection of
all fixators
of finite subcomplexes
form
a neighborhood basis
of the identity
in $\Aut X_.$.
These statements
are also true
for closed subgroups
of $\Aut X_.$.

Any closed subgroup,
$\overline{G}$ say,
of $\Aut X_.$,
which admits
a finite subcomplex
whose $\overline{G}$-translates
cover $X$,
is %in addition
compactly generated
and
cocompact in $\Aut X_.$.
\end{proposition}

Complete Kac-Moody groups
hence
have all the properties
described above,
including
compact generation
and
co-compactness
in the full automorphism group
of its building
even though
we will not use
the latter
two properties
in this paper.

\section{Geometric reformulation of topological group concepts}
\label{sec:geometrization(problem)}
In what follows,
we reformulate
topological group concepts
in geometric terms,
that is
in terms of
the action
on the building.
We begin with
a geometric reformulation
of relative compactness.

A closed subgroup
$\overline{G}$                                                  %changed
of the automorphism group
of a connected, locally finite, metric complex $X$
carries
two natural structures
of bornological group.

The first bornological group structure
on $\overline{G}$
is
the natural bornology
induced by
its topological group structure,
and
consists of
the collection of
all \emph{relatively compact} subsets
of the group $\overline{G}$.

The second bornological group structure
on $\overline{G}$
is
the bornology
induced by
the natural bornology on
the metric space $X$,
in which subsets of $X$ are bounded
if and only if
they have finite diameter;
this bornology
on the group $\overline{G}$
consists of
the collection of
subsets $M$
of $\overline{G}$
which have the property that
for every bounded subset $B$
of $X$
the set $M.B$ is also bounded.
One can verify that
the latter condition
on the subset $M$
of $\overline{G}$
is equivalent to
the condition that
for some, and hence any, point            %changed
$x$
of $X$
the set $M.x$ is bounded.
We will call
the sets
in the second bornology
on the group $\overline{G}$
\emph{bounded} sets.

We now  verify                              %changed
that these two bornologies
coincide.
For subsets $Y$, $W$                               %changed
of  the metric space $X$ define \quad
%\[
$
\Trans YW_{\overline{G}}.:=\{g\in \overline{G}\colon g.Y\subseteq W\}
$.
%\]
%
Note that
\[
\Trans {\{y\}}{\{w\}}_{\overline{G}}.=
\begin{cases}
g_{wy}\overline{G}_y=\overline{G}_wg_{wy}=
\overline{G}_wg_{wy}\overline{G}_y &
\text{if }\exists\, g_{wy}\in \overline{G}\colon g_{wy}.y=w\\
\varnothing& \text{else}
\end{cases}
\]
Hence,
whenever $\overline{G}$ is
a closed subgroup of
the automorphism group
of a connected, locally finite complex $X$
and $y$, $w$ are
points of $X$,
the set  $\Trans {\{y\}}{\{w\}}_{\overline{G}}.$
will be compact and open.

\begin{lemma}[geometric reformulation of `relatively compact']\label{lem:rel-compact_geo}
Let
$X$ be
a connected, locally finite, metric complex,
and assume that
$\overline{G}$ is
a closed subgroup of $\Aut X_.$
equipped with the compact-open topology.
Then
a subset of $\overline{G}$ is relatively compact
if and only if it is bounded.
\end{lemma}
\begin{proof}
We will use
the criterion
that a subset $M$ of $\overline{G}$ is bounded
if and only if,
for some chosen vertex,
$x$ say,
the set $M.x$
is  bounded.

Assume first that
$M$ is
a bounded subset of $\overline{G}$.
%
%By definition
%of the term bounded,
This means that
$M.x$ is a bounded,
hence finite
set of vertices.
We conclude that
\[
M\subseteq \bigcup_{y\in M.x}\Trans {\{x\}}{\{y\}}_\overline{G}.\,,
\]
which shows that
$M$ is
a relatively compact subset
of $G$.

Conversely,
assume that
$M$ is a relatively compact subset
of $\overline{G}$.
We have
\[
M\subseteq \bigcup_{y\in X}\Trans {\{x\}}{\{y\}}_\overline{G}.\,.
\]
and,
since $M$ is relatively compact,
there is a finite subset $F(M,x)$
of $X$
such that
\[
M\subseteq \bigcup_{y\in F(M,x)}\Trans {\{x\}}{\{y\}}_\overline{G}.=:T(M,x)\,.
\]
We conclude that
$M.x\subseteq T(M,x).x\subseteq F(M,x)$
which shows that
$M$ is bounded.
\qed
\end{proof}

\subsection{Geometric reformulation of topological properties of isometries}\label{subsec:isometries}
Under the additional condition
that the complex $X$
carries a $\CAT0$-structure,
we use
the previous result
to reformulate
the topological condition
on a group element
to be (topologically) periodic
in dynamical terms.

\begin{lemma}[weak geometric reformulation of `topologically periodic']\label{lem:elements_geo}
Let
$X$ be
a connected, locally finite, metric $\CAT0$-complex.
Equip $\Aut X_.$ with
the compact-open topology
and
let $g$ be
an element of $\Aut X_.$.
Then
$g$ is topologically periodic
if and only if
$g$ has a fixed point.
\end{lemma}
\begin{proof}
By Lemma~\ref{lem:rel-compact_geo},
$g$ is topologically periodic
if and only if
the group generated by $g$
is bounded.
Since
a bounded group of automorphisms
of a complete $\CAT0$-space
has a fixed point,
topologically periodic elements
have fixed points.

Conversely,
if $g$ fixes the point $x$ say,
then $g$,
and the group it generates,
is contained in
the compact set $\Aut X_._x$.
Hence
$g$ is topologically periodic.
\qed
\end{proof}

One can even
detect
the property of
being topologically periodic
in a purely geometric way:
isometries of $\CAT0$-spaces
which do not have fixed points
are either
parabolic or hyperbolic.
If,
in the previous lemma,
we impose
the additional condition
that
the complex $X$
should have
finitely many isometry classes of cells,
then
$X$
is known to have
no parabolic isometries
and
we obtain the following
neat characterization.

\begin{lemma}[strong geometric reformulation of `topologically periodic']
\label{lem:elements_geo+}
Let
$X$ be
a connected, locally finite, metric $\CAT0$-complex
with finitely many isometry classes of cells.
Equip $\Aut X_.$ with
the compact-open topology
and
let $g$ be
an element of $\Aut X_.$.
Then the following properties are equivalent:
\begin{enumerate}
\item
$g$ is topologically periodic;
\item
$g$ has a fixed point;
\item
$g$ is not hyperbolic.
\end{enumerate}
\end{lemma}
\begin{proof}
The assumption that
the complex $X$
has finitely many
isometry classes
of cells
implies that
no isometry of $X$
is parabolic
by a theorem of Bridson \cite[II.6.6 Exercise~(2) p.~231]{GMW319}.
This shows that
the second and third statement
of the lemma
are equivalent.
The first and the second statement
are equivalent
by Lemma~\ref{lem:elements_geo},
which concludes the proof.
\qed
\end{proof}

In the case
of interest to us,
we can add
a further characterization
of `topologically periodic'
to those
given above
and we include it for completeness                    %changed
although we will not need to use it.
The scale referred to in the statement
is defined as in~\cite{tdlcG.structure}                 %changed
and~\cite{furtherP(s(tdG))}.

\begin{lemma}[scale characterization of `topologically periodic']
\label{lem:top-periodic=scale1}
If $\overline{G}$ is
the geometric completion
of a group with locally finite twin root datum
(or the full automorphism group of its building)
%one can show that
the following statements
are also equivalent to
the statements~(1)--(3) of Lemma~\ref{lem:elements_geo+}:
\begin{enumerate}
\setcounter{enumi}{3}
\item
the scale value $s_{\overline{G}}(g)$ is equal to $1$;                 %changed
\item
the scale value $s_{\overline{G}}(g^{-1})$ is equal to $1$;            %changed
\end{enumerate}
Furthermore,
$s_{\overline{G}}(g)=s_{\overline{G}}(g^{-1})$
for all $g$ in $\overline{G}$.
\end{lemma}
\begin{proof}
This statement
follows form
Corollary~10 and Corollary~5
in~\cite{flat-rk(autGs(buildings))}.
\qed
\end{proof}

\subsection{Geometric reformulation of the topological definition of a contraction group}
\label{subsec:contrGs}

It follows from Lemma~\ref{lem:top-periodic=scale1}            %changed
and Proposition~3.24 in~\cite{contrG+scales(AUT(tdlcG))}
that in the geometric completion of a
group with locally finite twin root datum
contraction groups of topologically periodic elements
are bounded
while the contraction groups of elements
which are not topologically periodic are
unbounded.
In particular this observation applies
to topological Kac-Moody groups.

The following lemma
explains why
in this paper
we focus on                                          %changed
contraction groups
of non-topologically periodic elements.
Note that
we relax notation
and denote
the contraction group
of inner conjugation
with $g$
by $U_g$.

\begin{lemma}[contraction group of a topologically periodic element]
\label{lem:contrGs(top-periodic-elements)}
Suppose that
$g$ is
a topologically periodic element
of a locally compact group.
Then the contraction group $U_g$
is trivial
and hence closed.
\end{lemma}
\begin{proof}
This is
a special case
of Lemma~3.5 in~\cite{contrG+scales(AUT(tdlcG))}
where
$v=g$ and $d=e$.
\qed
\end{proof}

Membership in contraction groups
can be detected by
examining
the growth of
fixed point sets
while going to infinity.
The precise formulation is
as follows.

\begin{lemma}[geometric reformulation of  `membership in  a contraction group']
\label{lem:contractionGs_geo}
Let
$X$ be
a connected, locally finite, metric $\CAT0$-complex.
Equip $\Aut X_.$ with
the compact-open topology.
Suppose that
$h$ is
an hyperbolic isometry
of $X$
and
let $-\xi$ be
its repelling fixed point
at infinity.
%
%Further,
Let $l\colon \mathbb{R}\to X$ be a geodesic line
with $l(\infty)=-\xi$.

Then
an isometry $g$
of $X$
is in $U_h$
if and only if
for each $r>0$
there is
a real number $p(g,r)$
such that
all points
in $X$
within distance $r$
of the ray $l([p(g,r),\infty))$
are fixed
by $g$.
\end{lemma}
\begin{proof}
The assumption $l(\infty)=-\xi$
implies that
we may assume
without loss of generality
that
$l$ is an axis
of $h$.

Suppose now that
$g$ is
an isometry
of $X$
and
let $r(g,n)$ be
the radius of
the ball
around $P(g,n):=h^{-n}.l(0)$
that is fixed
by $g$,
with the convention that
$r(g,n)$ equals $-\infty$
if
$g$ does not fix
the point $P(g,n)$.
By
the definition of
the contraction group $U_h$
and
the topology
on $\Aut X_.$
the element $g$
is contained
in $U_h$
if and only if
$r(g,n)$ goes to infinity
as $n$ goes to infinity.

Since
$g$ is an isometry
and
$l$ is an axis
of $h$,
the points $P(g,n)$
for $n$ in $\mathbb{N}$
are equally spaced
along $l(\mathbb{R})$.
Therefore
we may reformulate
the condition
for membership
in $U_h$
given
at the end of
the last paragraph
as
in the statement of
the lemma.               %changed
\qed
\end{proof}

The results in~Lemma\ref{lem:contrGs(top-periodic-elements)},            %changed
Lemma~\ref{lem:elements_geo+}
and Lemma~\ref{lem:contractionGs_geo}
imply
the following dichotomy
for contraction groups.

\begin{lemma}[dichotomy for contraction groups]
\label{lem:dichotomy_contrGs}
If $X$ is
a connected, locally finite, metric $\CAT0$-complex
with finitely many isometry classes of cells
then
we have
the following dichotomy
for contraction groups
associated to isometries of~$X$.
\begin{itemize}
\item
Either                              %changed
the isometry
%under consideration
is elliptic
and its contraction group
is trivial,
\item
or
the isometry
%under consideration
is hyperbolic
and its contraction group
is the set of isometries
whose fixed point set grows
without bounds
when one approaches
its repelling fixed point
at infinity
as described in Lemma~\ref{lem:contractionGs_geo}.
\end{itemize}
\end{lemma}

\subsection{Geometric reformulation of the topological definition of a parabolic group}
\label{subsec:parabolics}
Using
the compatibility result
between the natural bornologies
in Lemma~\ref{lem:rel-compact_geo}            %changed
we can also prove
a geometric characterization
for membership
in parabolic groups.
We again relax notation            %changed
and denote
the parabolic group
of inner conjugation
with~$g$
by~$P_g$.

\begin{lemma}[geometric reformulation of `membership in a parabolic group']
\label{lem:parabolics_geo}
Let
$X$ be
a connected, locally finite, metric $\CAT0$-complex.
Suppose that
$h$ is
a hyperbolic isometry
of $X$
and
let $-\xi$ be
its repelling fixed point
at infinity.
Then
$P_h$ is
the stabilizer of $-\xi$.
\end{lemma}
\begin{proof}
Suppose first that
$g$ is
an element of
$P_h$.
Let $o$ be a point
of $X$.
% on an axis of $x$.
By
our assumption on $g$
and
by Lemma~\ref{lem:rel-compact_geo}
there is
a constant $M(g,o)$
such that
\[
d(h^ngh^{-n}.o,o)=d(g.(h^{-n}.o), (h^{-n}.o)) < M(g,o)\ \text{for all}\ n\in\mathbb{N}\,.
\]
But
the point $-\xi$
is %defined as
the limit of
the sequence $(h^{-n}.o)_{n\in \mathbb{N}}$
and
thus
by the definition
of points at infinity
of $X$
we infer
%from this condition            %changed
that
$g$ fixes $-\xi$.

Conversely,
assume that
$g$ fixes
the point $-\xi$.
The above argument
can be reversed
and
then shows that
$g$ is
contained in $P_h$.            %changed
%the proof is complete.
\qed
\end{proof}

There is
a dichotomy
for parabolic groups
that is
analogous to
the dichotomy
for contraction groups
obtained in Lemma~\ref{lem:dichotomy_contrGs};
the statement
is as follows.

\begin{lemma}[dichotomy for parabolic groups]
\label{lem:dichotomy_parabolics}
If $X$ is
a connected, locally finite, metric $\CAT0$-complex
with finitely many isometry classes of cells
then
we have
the following dichotomy
for parabolic groups
associated to isometries of $X$.
\begin{itemize}
\item
Either            %changed
the isometry
%under consideration
is elliptic
and its parabolic group
is the ambient group,
\item
or
the isometry
%under consideration
is hyperbolic
and its parabolic group
is the stabilizer
of its repelling fixed point
at infinity.
\end{itemize}
\end{lemma}
\begin{proof}
Applying
Lemma~3.5 in~\cite{contrG+scales(AUT(tdlcG))}
in the case
of  parabolic groups
with
$v=g$ and $d=e$
one sees that
parabolic groups
defined by
topologically periodic elements
are equal to
the ambient group;
this settles
the first possibility
listed above.
By Lemma~\ref{lem:elements_geo+}
an isometry that is not elliptic
must be hyperbolic
and then                             %changed
the parabolic group
has the claimed form
by Lemma~\ref{lem:parabolics_geo}.
\qed
\end{proof}

We conclude
this section
with the following
remark.

\begin{remark}
\label{rem:subgroup-contractionGs+parabolics_geo}
Suppose $G$ is a topological group,
$\mathfrak{g}\in\Aut G_.$
and $H$ is a $\mathfrak{g}$-stable
subgroup of $G$.
Then the contraction group of
$\mathfrak{g}$ in $H$ is
the intersection of the contraction group
of $\mathfrak{g}$ in $G$ with $H$;
an analogous statement 
is true for %UDO: changed `can be made about' to `is true for' 
the parabolic groups of $\mathfrak{g}$ 
%UDO: added 
within $H$ 
and $G$.
Thus
the geometric characterizations
of contraction groups and parabolics
given in
Lemmas~\ref{lem:contractionGs_geo} and~\ref{lem:parabolics_geo}
and
the dichotomies
described in
Lemma~\ref{lem:dichotomy_contrGs} and~\ref{lem:dichotomy_parabolics}
also hold for
subgroups of
$\Aut X_.$
for the specified
spaces $X$.
\end{remark}

\section{Outline of the proof of the Main Theorem}
\label{sec:proof-outline}

We know
from Lemma~\ref{lem:contrGs(top-periodic-elements)}
that contraction groups
of topologically periodic elements
are trivial
and hence closed.
This %already
proves
statement~1
of our Main Theorem.

Under the additional condition
on the type
of the Weyl group
given in statement~2,
%of our Main Theorem
we will show that
for any
non-topologically periodic element,
$h$ say,
of $\overline{G}$
%(equivalently,
%which induces a hyperbolic isometry)
the group $U_h\cap U_{h^{-1}}$
contains
a $\overline{G}$-conjugate
of a root group
from the natural root datum
for $G$.

\subsection{The criterion implying non-closed contraction groups}%
\label{subsec:criterion-non-closed-contrGs}
Theorem~3.32 in~\cite{contrG+scales(AUT(tdlcG))}
gives 12 equivalent conditions
for a contraction group
in a metric totally disconnected, locally compact group
to be closed.
By the equivalence of
conditions~(1) and~(4)
from Theorem~3.32 in~\cite{contrG+scales(AUT(tdlcG))}     %changed
the group $U_h$
is not closed
if and only if
the group $\overline{U}_{h}\cap  \overline{U}_{h^{-1}}$
is not trivial,
hence
the property
whose verification
we announced
in the previous paragraph
confirms statement~2
of our Main Theorem.
The proof of
this strengthening of
statement~2
of Theorem~\ref{thm:MainThm}
%which was mentioned
%in the introductory paragraphs
%to Section~\ref{sec:proof-outline}
proceeds in                        %changed
three steps.
\begin{enumerate}
\item
Firstly,
we show that
any geodesic line,
$l$ say,
can be moved
to a line $l'=g.l$
with image
in the standard apartment
by
a suitable element
$g$
of the completed group $\overline{G}$.
In what follows
we will be interested
only
in the case where
the line
$l$                                %changed
is an axis
of a hyperbolic isometry
$h\in\overline{G}$.                 
\item
Secondly,
we use %UDO: omitted `will'
the assumption
on the type of the Weyl group
to show that
for any geodesic line $l'$
in the standard apartment
there is
a triple of roots $(\alpha,\beta,\gamma)$
in ``fundamental hyperbolic configuration'' 
%UDO: added 
with respect to $l$.
By this we mean that                   
$\alpha,\beta$ and $\gamma$
are
pairwise non-opposite
pairwise disjoint roots,
such that
the two ends
of $l'$
are contained in
the respective interiors
of $\alpha$ and $\beta$. 
\item
Thirdly
and finally,
we use that
every
split or almost split
Kac-Moody group
has (uniformly)
contractive root groups,
a notion
introduced in Subsection~\ref{subsec:contr-rootGs} below,
to arrive at
the announced conclusion.
More precisely,                     
the geometric criterion
for membership
in contraction groups
is used %UDO: changed `can be' to `is' 
to show that
whenever
$h'$ is a hyperbolic isometry
in $\overline{G}$,
the line $l'$
is
an axis of $h'$
contained in
the standard apartment
and
the fundamental hyperbolic configuration $(\alpha,\beta,\gamma)$
is chosen
as mentioned
in the previous item,
then
the root group $U_{-\gamma}$
is contained in
the group $U_{h'}\cap U_{{h'}^{-1}}$.
\end{enumerate}

In terms of
the originally chosen
hyperbolic isometry $h$
and the element $g$
of $\overline{G}$
found
in step~1
above,
the conclusion
arrived at
after step~3
is that
$g^{-1}U_{-\gamma}g\subseteq U_h\cap U_{h^{-1}}$.

For our proof
to work,
we do not need
to assume
that
our original group $G$
is the abstract Kac-Moody group
over a finite field.
Step~1 
uses that
the group
is a completion of
a group with
a locally finite twin root datum,
Step~2 
uses a property
of the corresponding Coxeter complex
and Step~3
works for
groups with
a locally finite twin root datum
whose root groups are
contractive,
a notion
which
we introduce
now.

\subsection{Contractive root groups}
\label{subsec:contr-rootGs}

As explained above,
the following condition
will play a central role
in the proof of
our Main Theorem.
In the formulation
of that condition,
we denote
the boundary wall
of
the half-apartment
defined by
a root $\alpha$
by $\partial\alpha$,
as is customary.

\begin{definition}%[contractive root groups]
\label{def:contr-rootGs}
Let $G$ be
a group with % a locally finite
twin root datum $(U_\alpha)_{\alpha\in\Phi}$.
We say that
$G$ has \emph{contractive root groups}
if and only if
for all $\alpha$ in $\Phi$
we have:
If $x$ is
a point
in
the half-apartment
defined by $\alpha$,
then
the radius of
the ball
around $x$
which
is fixed pointwise by
%the root group
$U_\alpha$
goes to infinity
as
the distance
of $x$
to $\partial\alpha$ %the boundary wall
%of
%the half-apartment
%defined by $\alpha$
goes to infinity.
\end{definition}

The natural system
of root groups
of any split or almost split Kac-Moody group
satisfies
a stronger, uniform version
of the condition
of contractive root groups,
which we introduce now.
This latter condition
was called
condition~(FPRS)
in \cite{simpl+superrig(twin-building-lattices)},
where it was shown
in Proposition~4
that
any split or almost split Kac-Moody group
satisfies it.

\begin{definition}
\label{def:uniformly-contr-rootGs}
Let $G$ be
a group with % a locally finite
twin root datum $(U_\alpha)_{\alpha\in\Phi}$.
We say that
$G$ has \emph{uniformly contractive root groups}
if and only if
for each point $x$
in the standard apartment
of the positive building
defined by
the given root datum
and all roots $\alpha$ in $\Phi$
whose corresponding half-apartment
contains $x$,
the radius of the ball
which
is fixed pointwise by
%the root group
$U_\alpha$
goes to infinity
as
the distance
of $\partial\alpha$ %the boundary wall
%of the half-apartment
%defined by $\alpha$
to $x$
goes to infinity.
\end{definition}

\begin{remark}
% why the name
%We chose
%to rename
%condition~(FPRS)
%as above
%because,
By Lemma~\ref{lem:contractionGs_geo},
for a group,
$G$ say,
with
twin root datum $(U_\alpha)_{\alpha\in\Phi}$,
which has contractive root groups,
for any root $\alpha$
the root group $U_\alpha$
is contained in
the contraction group
of any element $g$
of $G$
whose repelling point
at infinity
is
defined by
a geodesic ray
contained in
the interior of
the half-apartment
defined by $\alpha$.
The latter condition
will be instrumental
in showing
our main theorem.
\end{remark}

Abramenko and M\"uhlherr
constructed
an example
of a group with twin root datum
that does not have
uniformly contractive root groups.
However,
in that example
the effect
of fixed point sets
staying bounded
is obtained by
going towards infinity
along
a non-periodic path of
chambers.
Therefore,
it is not possible
to find
an automorphism
of the building
that translates
in the direction
of that path.

In discussions between the authors and       %changed
Bernhard M\"uhlherr
he asserted
%told us
that
a bound on
the nilpotency degree
of subgroups
of the group
with twin root datum
would imply that
fixed point sets
always grow
without bounds
along periodic paths.

\begin{remark}
It would be interesting
to define
and investigate
quantitative versions
of the notions
of contractive and uniformly contractive root groups
for groups
with locally finite twin root datum.
These quantitative versions
would specify
the growth
of the radius
of the ball
fixed by
a root group
as a function of
the distance of
the center of
that ball
from the boundary hyperplane.
We suspect that
this growth
might be linear
in all situations
if and only if
all contraction groups
of elements
in the geometric completion
of a group
with locally finite twin root datum
are closed.
\end{remark}

\section{Proof of the Main Theorem}

We will prove
the following
generalization
of our Main Theorem.

\begin{theorem}[strong version of the Main Theorem]
\label{thm:strongMainThm}
Let $G$ be a group
with a locally finite twin root datum
and $\overline{G}$ the closure of
the image of $G$
in the automorphism group
of its positive building.
Then the following are true:
\begin{enumerate}
\item
The contraction group
of any topologically periodic element
in $\overline{G}$
is trivial.
\item
If
the root groups
of $G$
are contractive
and the type of $G$
is irreducible
and
neither spherical nor affine
then
the contraction group
of any element
that is
not topologically periodic
in $\overline{G}$
is not closed.
\end{enumerate}
Furthermore
every element
of infinite order
in the Weyl group
of $G$
lifts to
a non-topologically periodic element
of $\overline{G}$;
in particular,
if % whenever
the Weyl group
of $G$
is not of spherical type,
then
the group $\overline{G}$
contains non-topologically periodic elements.
\end{theorem}

The proof
of this theorem
will be obtained from
several smaller results
as outlined
in Subsection~\ref{subsec:criterion-non-closed-contrGs}
above.
%
%As
%noted there,                %changed
By Lemma~\ref{lem:contrGs(top-periodic-elements)},
we only need to prove
statement~2
and the
existence statement
for non-topologically periodic elements.

The
%outlined                                      %changed
first step
towards the proof
of statement~2
of Theorem~\ref{thm:strongMainThm}
is provided by
the following proposition.

\begin{proposition}[geodesic lines can be moved to the standard apartment]
\label{prop:}
Let $G$ a group
with
locally finite twin root datum.
Denote
by $\overline{G}$ the geometric completion
of $G$
defined by
the given root datum,
by $X$
the Davis-realization of
the corresponding positive building
and
by $\mathbb{A}$
the corresponding standard apartment.

If $l$ is a geodesic line in $X$,
then there is
an element $g$
in $\overline{G}$
such that
$g.l(\mathbb{R})$
is contained in $|\mathbb{A}|$
and
intersects the fundamental chamber.
\end{proposition}
\begin{proof}
Since
the group $G$
acts transitively
on chambers,
there is
an element $g'$ in $G$
such that
$g'.l(\mathbb{R})$ intersects
the fundamental chamber $c_0\in \mathbb{A}$.
We therefore
may, and will,
assume that
$l(\mathbb{R})$ intersects $c_0$
from the outset.

Whenever $l$ leaves $\mathbb{A}$,
necessarily at a wall,
use elements of the corresponding root group $U_\alpha$
which fixes $c_0$
to `fold $l$ into $\mathbb{A}$' .     
This needs to be done
at increasing distance from $c_0$
along $l$
`on both sides',
leading to
an infinite product
of elements from root groups.
The sequence
consisting of
the partial products
of that infinite product
is contained in
the stabilizer of $c_0$,
which is a compact set.
Hence
that sequence
has a convergent subsequence,
which implies that
the infinite product
defined above
is convergent,
with limit
$g$ say.
By construction,
$g$
attains
the purpose
of the element
of the same name
in the statement
of the proposition
and
we are done.
\qed
\end{proof}

The second step
in the proof
of statment~2 of Theorem~\ref{thm:strongMainThm}
consists of
the following strengthening of
Theorem~14 in~\cite{simpl+superrig(twin-building-lattices)}.

\begin{theorem}[%existence of  a
a ``fundamental hyperbolic configuration'' exists w.r.t. any line]
\label{thm:fundamental-hyperbolic-configuration_line}
Let $\mathbb{A}$ be
a Coxeter complex,
whose type is
irreducible
and
neither spherical
nor affine.
Suppose that
$l\colon \mathbb{R}\to |\mathbb{A}|$ is
a geodesic line. % in $|\mathbb{A}|$.
Then
there is a triple of roots $(\alpha,\beta,\gamma)$
which are
pairwise disjoint
and
pairwise non-opposite
such that
for suitably chosen
real numbers $a$ and $b$
the rays $l(]-\infty,a])$ and $l([b,\infty[)$
are contained in
the interior of
the half-apartments
defined by $\alpha$ and $\beta$ respectively.           %changed
\end{theorem}
\begin{proof}
The line $l(\mathbb{R})$
must cut
some wall
of $\mathbb{A}$,
$H$ say.
One of
the two roots
whose boundary is $H$
contains
the ray $l(]-\infty,a])$
for sufficiently small $a$;
we name
that root $\alpha$.
%UDO: changed back 
Since
the Coxeter complex
is not of spherical type,
there is                                    
another wall
$H'$
which cuts $l$,
but not $H$.
Call $\beta$
the root
whose boundary is $H'$
and which
contains
the ray $l([b,\infty[)$
for sufficiently large $b$.
The existence of a root $\gamma$
as in the statement
is then assured
by Theorem~14 in~\cite{simpl+superrig(twin-building-lattices)},
which completes
the proof.
\qed
\end{proof}

The third and final step
in the proof
of statment~2 of Theorem~\ref{thm:strongMainThm}
is an immediate consequence of
our assumption
that root groups
are contractive
and
the geometric criterion
for membership
in contraction groups.

\begin{lemma}[non-triviality of intersection of opposite contraction groups]
\label{lem:}
Let $\overline{G}$ be
a group
which contains
the root groups
of a group with twin root datum
all of whose root groups
are contractive.
Assume that
$h\in \overline{G}$ is not topologically periodic
and
let $l$ be an axis of $h$.
If
$\gamma$ is a root
whose position
relative to $l$
is as described in the previous lemma,
then
$U_{-\gamma}\subseteq U_h\cap U_{h^{-1}}$.
Hence,
$U_h$ is not closed.
\end{lemma}
\begin{proof}
Since
the root group $U_{-\gamma}$
is contractive,
Lemma~\ref{lem:contractionGs_geo} ensures that             %changed
it is contained in
any contraction group $U_k$
with the property that
the repelling fixed point
of $k$
at infinity
is defined by
a ray
that is
contained in
the interior of
the half-apartment
defined by $-\gamma$.                                    %changed
%(by the geometric characterization
%for membership
%in contraction groups).
%
Both $h$ and $h^{-1}$
satisfy
this condition on $k$,
hence $U_{-\gamma}\subseteq U_h\cap U_{h^{-1}}$
as claimed.
Since
$U_{-\gamma}$ is not trivial,
we infer
from Theorem~3.32 in~\cite{contrG+scales(AUT(tdlcG))}
that $U_h$ is not closed.
%showing
%the second claim.                           %changed
\qed
\end{proof}

The following lemma
provides
the final statement
of %the strong version of the Main Theorem,
Theorem~\ref{thm:strongMainThm}
and thereby
concludes the proof of that theorem.

\begin{lemma}[existence of non-topologically periodic elements]
\label{lem:ex(non-top-periodic)}
Let $G$ be
a group
with a locally finite twin root datum
and
$\overline{G}$ the closure of
the image of $G$
in the automorphism group
of its positive building.
Then
every element
of infinite order
in the Weyl group
of $G$
lifts to
a non-topologically periodic element
of $\overline{G}$;
in particular,
if the Weyl group
of $G$
is not of spherical type,
then
the group $\overline{G}$
contains
non-topologically periodic elements.
\end{lemma}
\begin{proof}
Since
a Coxeter group
is torsion
if and only if
it is of spherical type,
the second claim
follows from the first.
In what follows,
we will show
that the lift
of an element
$w$ in
the Weyl group
is topologically periodic
if and only if
$w$ has finite order. %UDO: changed `had' to `has'

By Lemma~\ref{lem:elements_geo+},
an element,
$n$ say,
of $\overline{G}$
is topologically periodic
if and only if
its action
on the building,
$X$,
has a fixed point.
If
that element $n$
is obtained
as an inverse image
of an element,
$w$ say,
of the Weyl group,
it belongs to
the stabilizer of
the standard apartment $\mathbb{A}$.
Since
the Davis-realization $|\mathbb{A}|$
of the standard apartment
is
a complete, convex
subspace of the
complete $\CAT0$-space $X$,
using
the nearest-point projection
from $X$
onto $|\mathbb{A}|$,
we see that
the action of $n$
on $X$
has a fixed point
if and only if
its restricted action
on $|\mathbb{A}|$
has a fixed point.
The latter condition
is equivalent to
the condition that
the natural action
of $w$ on $|\mathbb{A}|$
has a fixed point.
Since this happens
if and only if
$w$ has finite order,
our claim is proved.
\qed
\end{proof}

%   The geometric characterization
%   of membership
%   in parabolic groups,
%   Lemma~\ref{lem:parabolics_geo},
%   in contraction groups,
%   Lemma~\ref{lem:contractionGs_geo},
%   together with
%   Theorem~3.32 in~\cite{contrG+scales(AUT(tdlcG))}
%   imply the following corollary.

%
%   \begin{corollary}
%   Let $G$ be
%   a group
%   with a locally finite twin root datum,
%   whose root groups
%   are contractive
%   and
%   $\overline{G}$ the closure of
%   the image of $G$
%   in the automorphism group
%   of its positive building.
%   %
%   Assume that
%   $h$ is
%   a non-topologically periodic element
%   in $\overline{G}$.
%   %
%   Then
%   there is
%   a non-trivial element
%   in $\overline{G}$
%   that fixes
%   the attracting
%   and
%   the repelling
%   fixed point
%   of $h$ at infinity
%   and
%   whose fixed point set
%   grows without bounds
%   when approaching
%   the repelling fixed point.
%   \qed
%   \end{corollary}

\section{The case of a disconnected Dynkin diagram}

The following two results                        %changed
may be used
to reduce the determination
of contraction groups
for elements
in arbitrary complete Kac-Moody groups
to the determination
of the contraction groups
in the factors
defined by
the irreducible components.
Their proofs
are left
to the reader.

\begin{lemma}[product decomposition for root data with disconnected diagram]
\label{lem:disconnectedtype}
Let $G$ be
a group
with a locally finite twin root datum
such that
the type of $G$
is the product
of irreducible factors
whose restricted root data
define groups $G_1$, \ldots $G_n$.
Denote
by $\underline{H}$
the quotient
of a group $H$
by its center.
Then
\[
\underline{G}\cong \underline{G}_1\times \cdots\times \underline{G}_n
\qquad\text{and}\qquad
\overline{G}\cong \overline{G}_{1}\times\cdots\times \overline{G}_{n}\,.
\]
as abstract, respectively topological, groups.
\qed
\end{lemma}

\begin{lemma}[contraction groups of elements in products]
\label{lem:contrGs(elements<prod)}
Let $\overline{G}_1,\ldots,\overline{G}_n$
be locally compact groups
and
$(g_1,\ldots,g_1)\in \overline{G}_1\times\ldots\times\overline{G}_n$.
Then
\[
U_{(g_1,\ldots,g_n)}=U_{g_1}\times\cdots\times U_{g_n}\,.
\]
\qed
\end{lemma}

We conjecture that
the contraction groups
for elements
in a complete Kac-Moody group
of spherical or affine type
are always closed.
Supporting evidence
for that conjecture
is provided by
the following proposition.

\begin{proposition}[contraction groups for spherical and known affine types]
\label{prop:contrGs(spherical,known-affine)}
Let $\overline{G}$ be
a totally disconnected,
locally compact group.
If
\begin{enumerate}
\item
\label{contrGs(spherical-type)}
either
$\overline{G}$ is
the geometric completion of
an abstract Kac-Moody group
of spherical type
over a finite field,
\item
\label{contrGs(affine-type)}
or
$\overline{G}$
is a topological subgroup
of the general linear group
over a local field,
\end{enumerate}
then
all contraction groups
of elements
in $\overline{G}$
are closed.
\end{proposition}
\begin{proof}
To show statement~\ref{contrGs(spherical-type)},
observe that
an abstract Kac-Moody group
of spherical type
over a finite field
is a finite group.
The associated
complete group,
$\overline{G}$,
is then finite too
and hence  is a discrete group,            
because its topology is Hausdorff.
Contraction groups
in a discrete group
are trivial,
and it follows
that
all contraction groups
of all elements in $\overline{G}$
are closed
if $G$
is of spherical type.

As noted in Remark~\ref{rem:subgroup-contractionGs+parabolics_geo},
we obtain
the contraction group
of an element
$h$
with respect to
a (topological) subgroup,
$H$
by intersecting
the contraction group
relative to the ambient group                        
with $H$.

Thus 
to establish statement~\ref{contrGs(affine-type)}
it is enough
to treat
the special case
of the general linear group
over a local field,
$k$ say.
Using
the same observation again
and noting that
$\mathrm{GL}_n(k)$
can be realized as
a closed subgroup of $\mathrm{SL}_{n+1}(k)$
via $g\mapsto \diag(g,\det(g)^{-1})$,
it suffices
to prove
statement~\ref{contrGs(affine-type)}
in the special case
of the group $\mathrm{SL}_n(k)$,
where $k$ is
a local field.
But contraction groups
of elements
in $\mathrm{SL}_n(k)$
have been shown
to be $k$-rational points
of unipotent radicals
of $k$-parabolic subgroups
in~\cite[Lemma~2]{elem:BTR+T}
as explained in
Example~3.13(1) %UDO: changed `Exercise' to `Example' 
in~\cite{contrG+scales(AUT(tdlcG))}; %UDO: changed `of' to `in' 
as such
they are Zariski-closed
and hence
closed in
the Hausdorff-topology
induced by
the field $k$.
This proves
statement~\ref{contrGs(affine-type)}
for the group $\mathrm{SL}_n(k)$,
and,
by the previous reductions,
in all cases.
\qed
\end{proof}

There are
complete Kac-Moody groups
of affine type
for which it is unknown
whether
the criterion
listed under item~\ref{contrGs(affine-type)}
of Proposition~\ref{prop:contrGs(spherical,known-affine)}
can be applied.
For example,
the complete Kac-Moody groups
defined by the generalized Cartan-matrices
$\left(\begin{array}{cc}2 & m \\ -1 & 2\end{array}\right)$
with integral $m<-4$
are of that kind.

%\begin{acknowledgements}
%\end{acknowledgements}

%% BibTeX users please use
%\bibliographystyle{spmpsci}
%%\bibliographystyle{alpha}
%   \bibliography{% name your BibTeX data base
%   contrGs<topKacMoodyGs,%
%   AAGruppen,%
%   short-Abk,%
%   Baeume,%
%   %GGtheory,%
%   %Gitter,%
%   NCRaeume,%
%   %arithmetic,%
%   buildings,%
%   %diverses,%
%   %growth,%
%   %measures,%
%   %qi,%
%   %quartett,%
%   %short-Abk,%
%   %tilings,%
%   topG%,
%   %types,%
%   %unclassified%
%   }
%   \end{document}

\end{document}